\documentclass[11pt]{article}
\usepackage[utf8]{inputenc}

\usepackage{amsmath}
\usepackage{amsthm}
\usepackage{mathtools}
\usepackage{amsfonts}
\usepackage{amssymb}
\usepackage{comment}
\usepackage{blkarray}
\usepackage{thmtools} %new math ambient - also for cross references

\usepackage{lineno}

\usepackage[normalem]{ulem}

\usepackage[margin=3cm]{geometry}

\usepackage[hidelinks]{hyperref} %with hidden color
\usepackage[capitalize, nameinlink]{cleveref}
\usepackage{tikz} 
\usepackage{tikz-cd}
\usetikzlibrary{ %snake arrows
	calc,%
	arrows,%
	shapes,
	shapes.geometric,
    positioning,
    fit,
} 

\tikzcdset{arrow style=tikz, diagrams={>=stealth}}
\usepackage[labelsep=quad,indention=10pt]{subfig}
\usepackage{graphicx}

\definecolor{Bnavy}{RGB}{0, 66, 128}
\definecolor{Bdust}{RGB}{140,179,217}
\definecolor{Bsugarpaper}{RGB}{198, 217, 236}
\usepackage{faktor}

\theoremstyle{plain}
\newtheorem{theorem}{Theorem}[section]

\newtheorem{proposition}[theorem]{Proposition}

\theoremstyle{definition}
\newtheorem{definition}[theorem]{Definition}

\newtheorem{point}[theorem]{}

\counterwithin{subsection}{section}
\usepackage[linesnumbered,vlined, ruled]{algorithm2e}
\usepackage{algpseudocode}

\usepackage{float}
\SetKwInput{KwIn}{I}
\SetKwInput{KwOut}{O}

\newenvironment{method}[1][htb]
{% Update algorithm name
	\begin{algorithm}[#1]%
	}{\end{algorithm}}

\DeclareMathOperator{\Img}{im}
\DeclareMathOperator*{\argmin}{arg\,min}
\DeclareMathOperator*{\argmax}{arg\,max}
\DeclareMathOperator{\dg}{deg}
\DeclareMathOperator{\ent}{ent}
\newcommand{\intS}[2][n]{I^{#1}[#2)}
\newcommand{\posr}{\left[0,\infty\right)}
\newcommand{\D}[1][n]{D^{#1}}
\newcommand{\Sp}[1][n]{S^{#1}}

\newcommand{\define}[1]{{\bf #1}}
\title{Decomposing filtered chain complexes: geometry behind barcoding algorithms}

\author{Wojciech Chach{\'o}lski, Barbara Giunti, Alvin Jin, Claudia Landi}

\date{}

\begin{document}

\maketitle

\begin{abstract}
In Topological Data Analysis, filtered chain complexes enter the persistence pipeline between the initial filtering of data and the final persistence invariants extraction.
It is known that they admit a tame class of indecomposables, called interval spheres. 
In this paper, we provide an algorithm to decompose filtered chain complexes into such interval spheres. 
This algorithm provides geometric insights into various aspects of the standard persistence algorithm and two of its run-time optimizations.
Moreover, since it works for any filtered chain complexes, our algorithm can be applied in more general cases. As an application, we show how to decompose filtered kernels with it.
\end{abstract}

\section{Introduction}

Starting from the 1990s, new methods of topological shape analysis arose \cite{delfinado, Ferri1993, frosini, robins}, and developed in what is nowadays known as {\em persistent homology} \cite{carlsson_topology_data, Edelsbrunner2002}, with applications varying from quantum physics \cite{pierro_homological_2018} to medicine \cite{belchi_lung_2018}, biology \cite{xia_persistent_2014}, and geology \cite{jiang_pore_2018} (see \cite{database} for an extended list of references to applications). 
The underlying idea is that homology captures shape features such as components and voids. 
Measuring their persistence along a filtration makes it possible to infer their significance.
A popular method to apply persistence to data sampled from an underlying shape is to build its {\em Vietoris--Rips simplicial filtration}. 
This construction comes with guarantees about the actual presence of the homological features detected by persistence in the underlying shape \cite{Edelsbrunner2002,oudot}. 

Recently, the authors of \cite{bcw} have proposed a method to study not only the non-contractible parts of the data but also the contractible ones by using directly {\em filtered chain complexes} without applying homology to them. 
Among the results of \cite{bcw}, the most relevant here is the structure theorem stating that filtered chain complexes can be written as direct sums of indecomposables called {\em interval spheres}. 
In plain words, this means that a filtered chain complex contains algebraic shadows of geometry in the form of algebraic disks and spheres glued together.
It is also interesting that filtered chain complexes do not arise only as chain complexes of filtered simplicial complexes. 
For example, the cover construction in \cite{bcw} approximates any tame parametrised chain complex by a filtered chain complex. 
So filtered chain complexes provide also invariants of (forward) commutative ladders \cite{commutative_ladders} and zigzag modules \cite{zigzag_carlsson}. 
Thus, an algorithm decomposing filtered chain complexes can be applied in various scenarios.

The main goal of this paper is to propose an algorithm based on geometry to decompose filtered chain complexes into direct sums of interval spheres. 
With this goal in mind, we first identify {\em quasi-minimal sets of generators} as a satisfactory trade-off between the complexity of the problem of pre-computing minimal sets of generators, which is equivalent to that of the decomposition into interval spheres itself, and more general sets of generators that may contain an arbitrarily large number of elements. 
In the case of a filtered chain complex built from a simplicial complex, simplices themselves form a quasi-minimal set of generators of the filtered chain complex. 
Then we show how to split interval spheres off a filtered chain complex by choosing appropriate pairs of generators and how this split reflects in the boundary matrix of the filtered chain complex. 
In particular, each split is performed by modifying and shrinking the original matrix, first via a row Gaussian elimination and then by deleting two rows and two columns. 
It is worth pointing out that our method does not rely on a definite ordering of generators and can process boundary matrices with rows and columns in any arbitrary order, thus leaving the possibility of choosing a preferred ordering for convenience.

Since our algorithm works for arbitrary filtered chain complexes, not only those arising from filtered simplicial complexes, we can apply it in novel settings. 
As a demonstration of this, we introduce the notion of a {\em filtered kernel}. 
This is a filtered chain complex given by the kernel of an epimorphism between two filtered simplicial complexes. 
We remark that filtered kernels are in general only partially related to kernel persistence \cite{ker-img-coker, morozov-thesis, casas2020distributing}, but arise naturally when studying, e.g., metric contractions. 
In this context, we determine a quasi-minimal set of generators for a filtered kernel and its corresponding boundary matrix. 
In this way, simply applying our decomposition algorithm produces the decomposition of filtered kernels.

As the third contribution of this paper, we provide geometrical insights on the standard persistence algorithm \cite{edels_harer}, viewed as a decomposition of filtered chain complexes, and its clear and compress optimisations \cite{clearcompress,twist}, viewed as peeling off interval spheres from a filtered chain complex. 
Indeed, from our decomposition of filtered chain complexes into interval spheres, we straightforwardly obtain the decomposition of the corresponding persistence module into interval modules \cite{carlsson_zomorodian_computing, oudot} simply forgetting all the intervals of length zero. 
Vice versa, since most of the persistence algorithms work in practice at the chain complex level, it is possible to recover the decomposition of filtered chain complexes from them.
Therefore, our method can be seen as a geometric interpretation of some of the persistence algorithms \cite{adams2014javaplex, ripser_paper, phat_paper, Edelsbrunner2002,Edel-Olsb, maria2014gudhi, morozov2007dionysus} of persistence and their operations on the boundary matrix. 
Combining this with the possibility of applying our algorithm to arbitrary filtered chain complexes, one can extend the application of persistence algorithms beyond filtered cell complexes.
\medskip

{\bf Outline.} In \cref{section_preliminaries}, we introduce the necessary background on filtered chain complexes, and we review the relevant persistence algorithms.
In \cref{section_split_reduction}, we present our decomposition algorithm for filtered chain complexes, and we show how to apply it to filtered kernels.
In \cref{sec_alg_analysis}, we highlight the connections to persistence algorithms and elucidate the geometry behind them. 
Finally, in \cref{sec_conclusions}, we discuss some further questions that may be interesting to explore. 

\section{Preliminaries}\label{section_preliminaries}

Throughout the work, $\mathbb{F}$ denotes a finite prime field of characteristic $p$ for a prime integer $p$.

\begin{point}\label{point_chain_complexes}
	{\em Chain complexes.}
	Let ${\mathbf N}=\{0,1,\ldots\}$ be the set of natural numbers. 
	A (non-negatively graded) chain complex of $\mathbb{F}$-vector spaces is a sequence of linear maps $C=\{\delta_{n+1}\colon C_{n+1}\to C_{n}\}_{n\in{\mathbf N}}$ of $\mathbb{F}$-vector spaces, called \define{differentials}, such that $\delta_{n}\delta_{n+1}=0$ for all $n$ in ${\mathbb N}$. 
	A chain complex $C$ is called \define{compact} if $\bigoplus_{n\in{\mathbf N}} C_n$ is finite dimensional as a vector space \cite{adamek_rosicky}. 
	This happens if and only if $C_n$ is finite dimensional for all $n$ and $C_n$ is trivial for all $n>N\in\mathbf{N}$.
	We denote by $\mathbf{ch}$ the category of compact chain complexes.
	Throughout this work, all considered chain complexes are compact.
	
	We call an \define{$n$-sphere}, or simply a \define{sphere}, the chain complex $\Sp$ that is zero in all degrees but in degree $n$, where it is isomorphic to $\mathbb{F}$.
	We call an \define{$n$-disk}, or simply a \define{disk}, the chain complex $\D$, given by 
	\begin{equation*}
	\D_{k} = 
	\begin{cases*}
	\mathbb{F}  & if  $k = n, n-1$ \\
	0           & otherwise
	\end{cases*},
	\qquad
	\delta_{k} =
	\begin{cases*}
	\mathbf{1} & if $k = n$ \\
	0   & otherwise
	\end{cases*}
	\end{equation*}
	Explicitly:
	\begin{equation*}
	\begin{tikzcd}[row sep=0.1cm]
	& \scriptstyle\cdots
	& \scriptstyle n+1
	& \scriptstyle n
	& \scriptstyle n-1
	& \scriptstyle n-2
	& \scriptstyle\cdots
	\\
	\Sp \cong
	& \cdots \ar[r]
	& 0 \ar[r]
	& \mathbb{F} \ar[r]
	& 0 \ar[r]
	& 0 \ar[r]
	& \cdots
	\\
	\D \cong
	& \cdots \ar[r]
	& 0 \ar[r]
	& \mathbb{F} \ar[r, "\mathbf{1}"]
	& \mathbb{F} \ar[r]
	& 0 \ar[r]
	& \cdots
	\end{tikzcd}
	\end{equation*}
	Since we consider only non-negatively graded chain complexes, $\D[0]\cong \Sp[0]$.
\end{point}

\begin{point}\label{point_homology}
	{\em Homology.}
	The following vector spaces are called respectively the $n$-th \define{cycles} and the $n$-th \define{boundaries} of $C$:
	\[
	Z_nC:=
	\begin{cases}
	X_0 & \text{ if }n=0
	\\
	\ker(\delta_{n}\colon C_{n}\to C_{n-1}) & \text{ if } n\geq 1
	\end{cases},
	\qquad
	B_nC:=\Img(\delta_{n+1}\colon C_{n+1}\to C_n)
	\]
	Since $\delta_{n}\delta_{n+1}=0$, the $n$-th boundaries $B_nC$ form vector subspaces of that of the $n$-th cycles $Z_nC$.
	The quotient $Z_nC/B_nC$ is called the $n$-th \define{homology} of $C$ and is denoted by $H_nC$.
	We write $ZC$, $BC$ and $HC$ to denote the chain complexes $\{Z_nC\}_{n\in {\mathbf N}}$, $\{B_nC\}_{n\in {\mathbf N}}$, and $\{H_nC\}_{n\in {\mathbf N}}$.
	Note that they all have trivial differentials.
	As an example, consider $\Sp$ and $\D$ defined in Paragraph~\ref{point_chain_complexes}.
	We have $H\D=0$, and $H_{n}\Sp=\mathbb{F}$ and $H_{k}\Sp=0$ for all $k\neq n$.
\end{point}

\subsection{Filtered chain complexes}\label{section_tame_ch}

The symbols $\posr$ and $\left[0,\infty\right]$ denote the ordered set of non-negative reals and extended non-negative reals, respectively. 
Functors of the form $X\colon \posr\to \mathbf{ch}$ are also referred to as \define{parametrised chain complexes}.
The value of $X$ at $t$ in $\posr$ is denoted by $X^{t}$, and $X^{s\leq t}\colon X^s\to X^t$ denotes the morphism in $\mathbf{ch}$ that $X$ assigns to $s\leq t$. 
The morphism $X^{s\leq t}$ is also referred to as the \define{transition morphism} in $X$ from $s$ to $t$.
The symbol $X_{n}$ denotes the $n$-th degree of $X$, for all $t\in\posr$, i.e.\ a sequence of vector spaces and maps between them.

\begin{definition}
	An increasing sequence $0=\tau_0<\cdots<\tau_k$ in $ \posr$ \define{discretises} $X\colon \posr\to \mathbf{ch}$ if $X^{s\leq t}\colon X^s\to X^t$ may fail to be an isomorphism only when there is $a\in \{1,\dots, k\}$ such that $s<\tau_{a}\leq t$.
	A parametrised chain complex $X\colon \posr\to \mathbf{ch}$ is called \define{tame} if there is a sequence that discretises it.
\end{definition}

We are especially interested in the following class of tame parametrised chain complexes.

\begin{definition}\label{def_cofibrant}
	A tame parametrised chain complex $X$ is called a \define{filtered chain complex} if, for all $s<t\in \posr$, the transition morphism $X^{s<t}$ is a monomorphism.
\end{definition}

The following definition provides an interesting class of filtered chain complexes, parametrised by a natural number $n$ and an element of $\Omega\coloneqq\{(s,e)\in [0,\infty)\times [0,\infty]\ |\ s\leq e\}$.

\begin{definition}\label{def_interval_sphere}
	Let $n\in\mathbf{N}$ and $\left(s, e\right)\in\Omega$.
	An (\define{$n$-dimensional}) \define{interval sphere} is a filtered chain complex $\intS{s,e}$ isomorphic to:
	\begin{equation*}
	\intS{s,e}^{t} =
	\begin{cases*}
	0     & if $t < s$ \\
	\Sp   & if $s \leq t < e$ \\
	\D[n+1] & if $e \leq t < \infty$
	\end{cases*},
	\quad
	\intS{s,e}^{u\leq t} =
	\begin{cases*}
	0\hookrightarrow 0            & if $u\leq t<s$ \\
	0\hookrightarrow\Sp           & if $u<s\leq t<e$ \\
	\mathbf{1}_{\Sp}                     & if $s\leq u<t<e$ \\
	\Sp\hookrightarrow\D[n+1]     & if $u<e\leq t$ \\
	\mathbf{1}_{\D[n+1]}                 & if $e\leq u< \infty$ 
	\end{cases*}.
	\end{equation*}
\end{definition}

Note that, in the previous definition, we allow $s=e$ for $e<\infty$.
\medskip

We now recall the decomposition theorem of filtered chain complexes from \cite{bcw}. 
We remark that the structure theorem for filtered chain complexes appeared in the literature in different settings in \cite{framed_morse, dualities, structural_filtered, related_work_decomposition}.

\begin{theorem}{\em \cite[Th.\ 4.2]{bcw}} \label{dec_theorem}
	Any filtered chain complex is isomorphic to a finite direct sum $\oplus_{i=1}^{l} \intS[n_{i}]{s_{i},e_{i}}$, where $l$ could possibly be $0$.
	Moreover, if $\oplus_{i=1}^{l}\intS[n_{i}]{s_{i},e_{i}}\cong \oplus_{j=1}^{l'}\intS[n'_{j}]{s'_{j},e'_{j}}$, then $l=l'$ and there is a permutation $\sigma$ of the set $\{1,\ldots, l\}$ such that $n_{i}=n'_{\sigma(i)}$, $s_{i}=s'_{\sigma(i)}$, and $e_{i}=e'_{\sigma(i)}$ for $1\leq i\leq l$.
\end{theorem}

\begin{point}\label{relation_int_sphere_int_modules}
	{\em Relation between filtered chain complexes and persistence module.}
	Applying homology to a filtered chain complex, we obtain classical persistence modules \cite{carlsson_zomorodian_computing}. 
	In particular, applying $H_n$ to an $n$-dimensional interval sphere, we obtain either the trivial persistence module or a so called \define{interval module} \cite{oudot}: $H_{n}\intS{s,s}\cong 0$ for $s=e$, and $H_{n}\intS{s,e}\cong\mathbb{I}_{\left[s,e\right)}$ for $s<e$. 
	In general, by the additivity of the homology functor, we can state the following result linking the decomposition of filtered chain complexes into interval spheres to the decomposition of persistence modules into interval modules.

\begin{proposition}\label{prop_int_sphere_int_module}
	Let $X$ be a filtered chain complex. 
	Any decomposition of $X$ into interval spheres induces the decomposition of the persistence module $H_nX$ into interval modules: 
	\[
	X\cong \bigoplus_{i=1}^{l} \intS[n_{i}]{s_{i},e_{i}} \implies
	H_nX\cong \displaystyle\bigoplus_{i \colon  n_i=n, \ s_i<e_i} \mathbb{I}_{\left[s_{i},e_{i}\right)}\, .
	\]
\end{proposition}

As a consequence, any algorithm for decomposing filtered chain complexes into indecomposables provides as a by-product another way of computing the barcode decomposition of the persistence module $HX$, simply by forgetting all the intervals $[s_j, e_j)$ for which $s_j=e_j$. 
\end{point}

\cref{dec_theorem} is an encouraging mathematical result, but if one aims to use it in computations, one needs to provide an appropriate translation from the mathematical formulation to feasible computer input. 
With this goal in mind, we first provide the right mathematical setting using the notion of generators and then restrict to a special class of them, the quasi-minimal set of generators, which correspond to a feasible input --- in fact, they correspond to the simplices in a simplicial complex.

\begin{point}\label{point_generators}
	{\em Generators.}
	Let $X$ be a filtered chain complex. 
	With a little abuse of notation, in what follows we write $\intS[-1]{s,s}$ for $\intS[0]{s,\infty}$.
	By tameness, there is a finite collection of interval spheres $\{\intS[n_{i}]{s_{i},s_{i}}\}_{i=1,\dots,m}$ that \define{generates} $X$, i.e.\ such that there is an epimorphism
	\begin{equation}\label{eq_epimorphism_generator}
	\Phi \colon \displaystyle\bigoplus_{i=1}^{m}\intS[n_{i}]{s_{i}, s_{i}} \to X \, .
	\end{equation}
	The morphism $\Phi$ uniquely defines $\Phi_{i}\colon\intS[n_{i}]{s_{i},s_{i}}\to X$,
	for all $i=1,\dots,m$. 
	In turn, for all $i=1,\dots,m$, each morphism $\Phi_{i}\colon\intS[n_{i}]{s_{i},s_{i}}\to X$ can be uniquely described by the vector $x_{i}\coloneqq(\Phi_{i})_{n_{i}+1}^{s_{i}}(1)\in X_{n_{i}+1}^{s_{i}}$ (see \cite[Paragraph 4.5]{bcw}).
	Therefore, $\mathcal{G}=\{x_{1},\dots, x_{m}\}$ is also called a \define{set of generators of $X$}.
	On $\mathcal{G}$, we can define a \define{degree} function $\dg$ and an \define{entrance time} function $\ent$ such that, given the generator $x=\Phi_{n}^{t}(1)$, they return $\dg(x)=n$ and $\ent(x)=t$. 
	A set of generators $\mathcal{G}$ is called \define{minimal} if there is no other set of generators $\mathcal{G}'$ with $\vert \mathcal{G}'\vert<\vert\mathcal{G}\vert$.
	
	Retrieving a minimal set of generators for a filtered chain complex is as hard as decomposing it into interval spheres. 
	So, it is more convenient to consider larger sets of generators.
	With this goal in mind, we recall that, for each $n\in\mathbf{N}$, the radical of $X_{n}$ is the subfunctor $\text{rad}X_{n}\subset X_{n}$ whose value $\text{rad}(X_{n})^{t}$ is the subspace of $X_{n}^{t}$ given by the sum of all the images of $X_{n}^{s<t}$ for all $s<t$. 
	A set of generators $\mathcal{G}$ of $X$ is called \define{quasi-minimal} if, for all $n\in\mathbf{N}$ and $t\in\posr$, the set 
	\[
	\{x \in \mathcal{G} \ \vert \ \ent(x)=t \ \text{ and } \dg(x)=n \}\]
	is a basis of $(\faktor{X_{n}}{\mathrm{rad}X_{n}})^{t}$.
	\medskip 
	
	As an example, take $X$ to be an interval sphere $\intS{s,e}$, with $n\in\mathbf{N}$ and $(s,e)\in\Omega$.
	A finite set of generators of $X$ is given by $\mathcal{G}=\{x_{n}^{s}, x_{n+1}^{e}\}$ with $x_{n}^{s}\in X_n^s\setminus\{0\}$, $x_{n+1}^{e}\in X_{n+1}^e\setminus\{0\}$ arbitrarily chosen.
	If $s<e$, then $\mathcal{G}$ is minimal. 
	If $s=e$, $\mathcal{G}$ is not minimal, since $x_{n}^{s}$ is redundant. 
	However, it is quasi-minimal.  
\end{point}

We use quasi-minimal sets of generators of $X$ to give a matrix encoding of $X$.

\begin{point}\label{point_filt_bound_matrix}
	{\em Total boundary matrix.}
	Let $X$ be a filtered chain complex with differential $\delta$, and $\mathcal{G}=\{x_{1}, \dots, x_{m}\}$ be an arbitrarily totally ordered quasi-minimal set of generators of $X$.  
	From the definition of radical, and the assumptions that $X$ is tame and all its transition morphisms are monomorphisms, it follows that, for each $n\in\mathbf{N}$ and $t\in\posr$, 
	\begin{equation}\label{eq_basis}
	X_{n}^t\cong \bigoplus_{s\leq t\in\posr} 
	\left(
	\faktor{X_{n}}{\mathrm{rad}X_{n}}
	\right)^{s}.
	\end{equation}
	
	By definition, a quasi-minimal set of generators $\mathcal{G}$ of $X$ provides a basis of $X_n^{t}$. 
	Thus, for each $x_j\in \mathcal{G}$ with $\dg(x_{j})=n+1$ and $\ent(x_{j})=t$, we can write $\delta_{n+1}^t(x_j)$ uniquely as a linear combination of elements $x_{i} \in \mathcal{G}$ such that $\dg(x_{i})=n$ and $\ent(x_{i})\leq t$.  
	The total boundary matrix of $X$ is the matrix $d =(d^i_j)$, $1\leq i,j \leq \vert \mathcal{G}\vert$, such that $d^i_j$ is equal to the coefficient of $x_i$ in $\delta_{n+1}^t(x_j)$.
	Hereafter, we adopt the notational convention by which superscripts denote row indices and subscripts column indices.
	Accordingly, $d^{i}$ denotes the $i$-th row and $d_{i}$ the $i$-th column of $d$.
	
	By construction, we can assume the rows and columns of $d$ are labelled by the degree and entrance time of the corresponding generator in $\mathcal G$.
	The sub-matrix of $d$ corresponding to the rows and columns associated with generators with degree $n-1$ and $n$, respectively, is the boundary matrix of the $n$-th differential $\delta_n$ of $X$. 
	In the case when the total order on $\mathcal{G}$ is given first by degree, $d$ is a diagonal block matrix with blocks given by the differential $\delta_n$.
	
	In practical situations, we often obtain the filtered chain complex $X$ from a filtered simplicial complex $\Sigma$. 
	In this case, the set of simplices of $\Sigma$ provides a quasi-minimal set of generators of $X$, and the total boundary matrix of $X$ coincides with the boundary matrix of $\Sigma$.
\end{point}

Since, in what follows, we constructively decompose a filtered chain complex by modifying its (quasi-minimal) set of generators, we now review in detail how the change of generators affects the corresponding boundary matrix.

\begin{point}\label{point_change_bound_matrix}
	{\em Change of generators.}
	Let $X$ be a filtered chain complex and $\mathcal{G}$ a totally ordered quasi-minimal set of generators of $X$ labelled by degree and entrance time. 
	Suppose now we change one of the generators $x_i\in \mathcal{G} $ to a new element 
	\[
	\tilde{x}_i= \alpha x_{i} + \displaystyle\sum_{j\neq i} \beta_{j}x_{j}
	\] 
	where $\alpha,\beta_j \in \mathbb{F}$ are such that $\alpha\neq 0$, and $\beta_{j}=0$ when $\ent(x_{j})>\ent(x_{i})$ or $\dg(x_{j})\neq\dg(x_{i})$. 
	Taking $\mathcal{\tilde{G}}=\{\tilde{x}_{1},\dots, \tilde{x}_{m}\}$ with $\tilde{x}_k=x_k$ for $k\ne i$, and $\tilde{x}_i$ as described, we claim that $\mathcal{\tilde{G}}$ is again a quasi-minimal set of generators of $X$.
	Indeed,	the change from $\mathcal{G}$ to $\mathcal{\tilde{G}}$ corresponds to a change of basis in $(\faktor{X_{n}}{\text{rad}X_{n}})^{t}$, with $n=\deg(x_i)$ and $t=\ent(x_i)$. 
	Hence, it is realised by an isomorphism $\psi\colon X\to X$, showing that $\mathcal{\tilde{G}}$ is also a set of generators of $X$.
	Moreover, in $\mathcal{\tilde{G}}$ the degree and the entrance time functions operate as follows: 
	\begin{align*}
	\dg(\tilde{x}_k) & =\dg(x_k) \qquad \ \text{ for all } k  =1,\dots,m \, ;
	\\
	\ent(\tilde{x}_j) & =\ent(x_j) \hspace{1.05cm} \text{ for all } j=1,\dots\hat{\imath}, \dots,m \, ;
	\\
	\ent(\tilde{x}_i) &=\max\{\ent(x_i),\max\{\ent(x_j) \ \vert \ j=1,\dots,\hat{\imath},\dots,m \text{ such that } \beta_{j} 
	\neq 0\} \} 
	\\
	&=\ent(x_{i}) \, .
	\end{align*}
	Thus, for each $n\in\mathbf{N}$ and $t\in\posr$, the number of generators with degree $n$ and entrance time $t$ is the same in $\mathcal{G}$ and $\mathcal{\tilde{G}}$, yielding that  $\mathcal{\tilde{G}}$ is quasi-minimal as well.
	
	The new total boundary matrix $\tilde{d}=(\tilde{d}_{j}^{i})$ of $X$ relative to $\mathcal{\tilde{G}}$ can be deduced from the original total boundary matrix $d=(d^i_j)$ relative to $\mathcal G$ by Gaussian elimination.
	In particular, the substitution of $x_i$ by $\tilde{x}_i= \alpha x_i$, with $\alpha\neq 0\in\mathbb{F}$, is realized by the row operation $d^i\to (\tilde{d})^{i}=(\alpha)^{-1} d^i$ and by the column operation $d_i\to (\tilde{d})_{i}=\alpha d_i$. 
	Similarly, the substitution of $x_i$ by $\tilde{x}_i= x_i+\beta x_k$, with
	$\beta\ne 0\in\mathbb{F}$ and $k\ne i$, is realized by the row operation $(\tilde{d})^{k}=d^k-\beta d^i$ and by the column operation $(\tilde{d})_{i}=d_i+\beta d_k$. 
	Note that, as degrees and entrance times of generators have not changed, the rows and columns of $\tilde{d}$ are still labelled with the same degrees and entrance times as in $d$, respectively.
\end{point}

\subsection{Persistence algorithms}\label{sec_standard_alg}

In view of relating our algorithm for filtered chain complexes decomposition to persistence algorithms, we recall the \define{standard persistence algorithm} \cite[Sec.\ VII.1]{edels_harer}, in short the SPA, and two of its runtime optimisations known, respectively, as clear and compress \cite{clearcompress, twist}.
\medskip

Let $d$ be the total boundary matrix of $X$ associated with the differential $\delta$ with respect to $\mathcal{G}$. 
In the SPA and its variants, the generators in $\mathcal{G}$ are sorted according to first entrance times, then degrees, and finally randomly to break the ties. 
This entails the fact that the row index of the lowest non-zero element in a column, called a \define{pivot}, corresponds to the youngest face of the column simplex. 
The boundary matrix is called \define{reduced} if there are no two non-zero columns with the same pivot.
Suppose $i$ is a pivot of column $j$. In that case, the generator corresponding to the $i$-th row gives birth to a homology class (hence it is said to be \define{positive}), while the generator corresponding to the $j$-th column causes the death of that homology class (hence it is said to be \define{negative}). 
The indices $\left(i, j\right)$ form a \define{persistence pair}.
The collection of all persistence pairs with $i<j$ of the reduced boundary matrix of $X$ provides the persistence diagram of $HX$, i.e. its barcode decomposition \cite{carlsson_zomorodian_computing, edels_harer, Edelsbrunner2002}.

The standard persistence algorithm \cite{edels_harer}, as well as its cohomological version \cite{dualities}, achieve such reduction using left-to-right column operations. 
The worst-case runtime complexity is cubic in the number of generators (see \cite{morozov_worst} for an example where the worst-case is attained). 
To obviate this problem, time optimisation strategies have been introduced.

\begin{point}\label{point_clear_compress_gen}
	{\em Clear and compress \cite{twist}.} 
	The \define{clear} optimisation is based on the observation that if $i$ is the pivot of a reduced column (i.e.\ it is positive), then during the reduction the $i$-th column will eventually turn out to be trivial. 
	Thus, instead of explicitly performing all the operations to make it trivial, it can be directly cleared, i.e.\ set to zero. 
	Symmetrically, the \define{compress} optimisation is based on the observation that if the $j$-th column contains a pivot (i.e.\ it is negative), then the row index $j$ cannot contain any pivot.
	Thus, one can avoid unnecessary computation and set the $j$-th row directly to zero. 

	Since the clear strategy requires processing columns in decreasing degree order while the compress runs in increasing degree order, they cannot be easily combined. 
	A technique to combine them is \cite{clearcompress}, where the clear is applied first, but only limited to pairs not exceeding a threshold persistence, followed by the compress on the remaining columns. 
    
    From the point of view of decomposing persistence modules, the clear and compress are clever runtime optimisations that take advantage of the property of boundary maps. 
    On the other hand, from the point of view of decomposing filtered chain complexes, the clear and compress have a geometrical explanation corresponding to the splitting of the pair of generators of an interval sphere (see Paragraph~\ref{point_standard}). 
\end{point}

\section{Decomposition method for filtered chain complexes}\label{section_split_reduction}

In this section, we present a computational strategy to decompose filtered chain complexes.
We begin by proving results ensuring the correctness of the algorithm, and then we provide its pseudo-code. 
We then run the algorithm on two examples coming from the Vietoris--Rips filtration of two point clouds to clarify it. 
Finally, we show how to apply the algorithm to filtered kernels in a situation where the filtered chain complex to be decomposed does not correspond to a filtered simplicial complex.
\medskip

Let $X$ be a filtered chain complex, and assume that we are given an arbitrarily totally ordered quasi-minimal set of generators ${\mathcal G}=\{x_1,x_2,\ldots, x_{m}\}$ of $X$.
We denote by $d=(d^i_j)$ the total boundary matrix associated with $\delta$ with respect to ${\mathcal G}$, i.e.\ $\delta(x_j)=\sum_{i=1}^m d^i_j x_i$.
All generators, and hence the corresponding rows and columns of $d$, are assumed to be labelled by degree and entrance time.
\medskip

A pair $(x_{i},x_{j})$ of generators in $\mathcal{G}$ is said to satisfy the \textbf{split conditions} \hyperlink{SC1}{\bf SC1-2-3} if 
\begin{equation*}
\textbf{SC\hypertarget{SC1}{1}: } d_{j}^{i}\neq 0; \qquad
\textbf{SC\hypertarget{SC2}{2}: } x_{i} \in \displaystyle\argmax_{h \ \mid \ d_{j}^{h}\neq 0} \ent(x_h);
\qquad
\textbf{SC\hypertarget{SC3}{3}: } x_j \in \argmin_{h\ \mid \ d_{h}^{i}\neq 0}\ent(x_h).
\end{equation*}
In plain words, $x_{i}$ is any of the possibly many elements with the smallest entrance time in the differential of $x_{j}$, and $x_{j}$ is any of the possibly many elements with the greatest entrance time whose image under $\delta$ contains $x_{i}$. 
As we will discuss in Paragraph~\ref{remark_apparent_pairs}, these conditions are related to the so-called \define{apparent} and \define{emergent pairs}.
\medskip

The reason for the name of the above conditions is that they allow us to split off an interval sphere (see \cref{prop_split}). 
Splitting interval spheres makes the differential more and more trivial. 
Thus, the strategy for the decomposition is to prune off elements to simplify the differential, similarly to what is done in \cite{gallais} for single chain complexes.
\medskip

If $\mathcal{G}$ is sorted as in the SPA (see \cref{sec_standard_alg}), the existence of a pair $(x_{i},x_{j})$ satisfying \hyperlink{SC1}{\bf SC1-2-3} is straightforwardly ensured, provided $d\ne 0$. 
Indeed, in that case, the generator $x_{j}$ can be taken to correspond to the first non-zero column of $d$ (which exists by the assumption $d\ne 0$), and the generator $x_{i}$ can be taken to correspond to the lowest non-zero element in such column. 

We now show that the existence of such a pair is still guaranteed, as long as there is at least one non-trivial boundary, even if the generators are ordered according to different ordering. 

\begin{proposition}\label{prop_existence_xi_xj}
	Let $X$ be a filtered chain complex and ${\mathcal G}=\{x_1,x_2,\ldots, x_{m}\}$ a quasi-minimal set of generators of $X$.
	Then there exists a pair $(x_{i},x_{j})$  of generators in $\mathcal{G}$ that satisfy the split conditions \hyperlink{SC1}{\bf SC1-2-3} if and only if there is at least one non-trivial differential.
\end{proposition}

\begin{proof}
    If there exists such a pair, then by \hyperlink{SC1}{\bf SC1} $X$ has at least one non-trivial differential.
    
    We now prove the other implication. 
    Since $X$ has at least one non-trivial differential, there exists an index $i_{0}\in\{1,\dots,m\}$ and a degree $n\in\mathbb{N}$ such that $x_{i_{0}}\in \Img(\delta_{n})$, implying the existence of $h$ such that $d_{h}^{i_{0}}\neq 0$.
	Choose $x_{j_{0}}$ in $\{\argmin_{h \mid d_{h}^{i_{0}}\neq 0} \ent(x_h)\}$.
	Observe that $\ent(x_{j_{0}})\geq \ent(x_{i_{0}})$. 
	Now, if $x_{i_0}\in \{\argmax_{h \mid d_{j_{0}}^{h}\neq 0} \ent(x_h)\}$, the claim is proved with $i=i_{0}$ and $j=j_{0}$.
	Otherwise, choose $x_{i_{1}}\in\{\argmax_{h \mid d_{j_{0}}^{h}\neq 0} \ent(x_h)\}$.
	We have $\ent(x_{i_{1}})>\ent(x_{i_{0}})$.
	If $x_{j_{0}}\in\{\argmin_{h \mid d_{h}^{i_{1}}\neq 0} \ent(x_h)\}$, the claim is proved with $i=i_{1}$ and $j=j_{0}$.
	Otherwise, choose $x_{j_{1}}\in\{\argmin_{h \mid d_{h}^{i_{1}}\neq 0} \ent(x_h)\}$ and note that $\ent(x_{j_{1}}) < \ent(x_{j_{0}})$.
	We iterate the process, selecting at the $k$-th step generators two generators such that one of the following happens: \textit{(i)} $x_{i_{k}}\in \{\argmax_{h \mid d_{j_{k}}^{h}\neq 0} \ent(x_h)\}$, so the claim is proved with $i=i_{k}$ and $j=j_{k}$, \textit{(ii)} $x_{j_{k-1}}\in \{\argmin_{h \mid d_{h}^{i_{k}}\neq 0} \ent(x_h)\}$, so the claim is proved with $i=i_{k}$ and $j=j_{k-1}$, \textit{(iii)} none of the above and the process iterates to step $k+1$.
	Since $X$ is tame, this procedure eventually ends, proving the claim.
\end{proof}

We are interested in a pair satisfying \hyperlink{SC1}{\bf SC1-2-3} because it determines generators of an interval sphere that can be split, thus shrinking a filtered chain complex $X$ to a smaller filtered chain complex $X'$. 
In the next proposition, we show how to identify the generators to be split using the boundary matrix of $X$ and retrieve the boundary matrix of $X'$ by a row Gaussian elimination on the boundary matrix of $X$.

\begin{proposition}\label{prop_split}
Let $X$ be a filtered chain complex with at least one non-trivial differential, and ${\mathcal G}=\{x_1,x_2,\ldots, x_{m}\}$ be a totally ordered quasi-minimal set of generators of $X$. 
Let $\left(x_{i},x_{j}\right)$ be a pair of generators in $\mathcal{G}$ satisfying the split conditions \hyperlink{SC1}{\bf SC1-2-3}. 
Then there exists a quasi-minimal set of generators $\tilde{\mathcal{G}}=\{\tilde{x}_1,\dots,\tilde{x}_{m}\}$ of $X$ for which
\begin{enumerate}
\item\label{claim_1} $\tilde{x}_{j}=x_j$, $\tilde{x}_{i}=\delta(x_j)$, and the pair $\left(\tilde{x}_{i},\tilde{x}_{j}\right)$ satisfies the split conditions \hyperlink{SC1}{\bf SC1-2-3};
\item\label{claim_2} Each $\tilde{x}_k\in \tilde{\mathcal{G}}$ is labelled with the same degree and entrance time as $x_k\in {\mathcal G}$, $k=1,\dots,m$;
\item\label{claim_3} $X\cong\intS{s,e}\oplus X'$, with $n=\dg(\tilde{x}_i)$, $s=\ent(\tilde{x}_i)$, $e=\ent(\tilde{x}_j)$, and $X'$ is the filtered chain complex generated by $\mathcal{G'}=\tilde{\mathcal{G}}\setminus\{\tilde{x}_i,\tilde{x}_{j}\}$ with differential $\delta'$ induced by $\delta$;
\item\label{claim_4} The $(m-2)\times(m-2)$ total boundary matrix $d'$ of $X'$ with respect to $\mathcal{G'}$ is obtained from the $m\times m$ total boundary matrix $d$ of $X$ with respect to $\mathcal{G}$ by adding to each of its $k$-th row with $d^k_j\ne 0$ the $i$-th row multiplied by $-(d^i_j)^{-1}d^k_j$, and next removing its $i$-th and $j$-th columns, as well as its $i$-th and $j$-th rows.
\end{enumerate}
\end{proposition}

\begin{proof}
Recall from \cite[Thm. 4.2]{bcw} that a sufficient condition to split an interval sphere $\intS{s,e}$ out of a filtered chain complex $X$ with non-trivial differential $\delta$ is to take two non-zero elements $x\in X_{n-1}^s$ and $y\in X_n^e$ that satisfy the following properties: 

	\begin{alignat*}{3}\label{prop:interval}
	&\text{(i)} && \delta_{n-1}(x)= 0 \, ; && \\
	&\text{(ii)} && \delta_n(y)=X_{n-1}^{s\le e}(x) \, ; && \\
	&\text{(iii)} \ && x\not\in \Img(X_{n-1}^{t<s}) \qquad \quad && \text{ for any } t<s\, ; \\
	&\text{(iv)} && y\not\in \Img(X_n^{t<e})  && \text{ for any } t<e \, ; \\ 
	&\text{(v)} \ && X_{n-1}^{s<t}(x) \not\in \Img(\delta_n^t) && \text{ for any } t<e \, .
	\end{alignat*}
	Under these assumptions, $X$ is isomorphic to $\intS{s,e}\oplus X'$ for a filtered chain complex $X'$ uniquely determined up to isomorphism by quotienting out $x$ and $y$. 

The existence of a pair of generators $\left(x_{i},x_{j}\right)$  of $\mathcal{G}$ satisfying the split conditions \hyperlink{SC1}{\bf SC1-2-3}, ensured by \cref{prop_existence_xi_xj}, is not enough to guarantee property (ii). 
Let us therefore construct a possibly different set of quasi-minimal generators $\tilde{\mathcal G}\coloneqq\{\tilde{x}_{1},\dots, \tilde{x}_{m}\}$ so that the pair $\left(\tilde{x}_{i},\tilde{x}_{j}\right)$ satisfies properties (i--v).

The first step is to define   $\mathcal{\overline{G}}\coloneqq\{\overline{x}_{1},\dots, \overline{x}_{m}\}$ with
	\begin{equation}\label{set_generators_bar}
		\overline{x}_{k} \coloneqq \left\lbrace
		\begin{array}{lll}
		\displaystyle\sum_{r} d_{j}^{r}x_{r} & \mbox{ if $k=i$}
		\\
		x_{k} & \mbox{ otherwise}
		\end{array}
		\right. 
		\, .
	\end{equation}
In other words, we take $\overline{x}_{i}$ to be the differential of $x_{j}$. 
Let $\overline{d}=(\overline{d}_{h}^{k})$ be the matrix of $\delta$ with respect to $\mathcal{\overline{G}}$. 
By the discussion in Paragraph~\ref{point_change_bound_matrix}, $\overline{\mathcal{G}}$ is a quasi-minimal set of generators.
Moreover, $\dg(\overline{x}_{k})=\dg(x_{k})$ for all $k$, and $\ent(\overline{x}_{k})=\ent(x_{k})$ for all $k$ since $\ent(\overline{x}_{i})=\max\{\ent(x_{k})\mid d_{j}^{k}\neq 0\}=\ent(x_{i})$ by \hyperlink{SC2}{\bf SC2}.
	
The second step is to possibly change again the set of generators to $\mathcal{\tilde{G}}\coloneqq\{\tilde{x}_{1},\dots, \tilde{x}_{m}\}$ by putting $\tilde{x}_{k}\coloneqq\overline{x}_{k}$ for $k=j$, and $\tilde{x}_{k}\coloneqq \overline{x}_{k} - \overline{d}_{k}^{i}\overline{x}_j$ otherwise.
$\mathcal{\tilde{G}}$ is still a quasi-minimal set of generators, again by the discussion of Paragraph~\ref{point_change_bound_matrix}, and $\dg(\tilde{x}_{k})=\dg(\overline{x}_{k})$ and $\ent(\tilde{x}_{k})=\ent(\overline{x}_{k})$ for all $k$.
Moreover, by construction, also $(\tilde{x}_{i},\tilde{x}_{j})$ satisfies \hyperlink{SC1}{\bf SC1-2-3}. 
So, claims \textit{\ref{claim_1}} and \textit{\ref{claim_2}} are verified. 

To prove claim \textit{\ref{claim_3}}, 
it is sufficient to show that properties (i--v) are satisfied by taking $x:=\tilde{x}_{i}$ and $y\coloneqq\tilde{x}_{j}$. 
Property (i) follows by \hyperlink{SC1}{\bf SC1-2} of $(\tilde{x}_{i}, \tilde{x}_{j})$. 
Property (ii) holds by construction since $\delta_n(\tilde{x}_{j})=\tilde{x}_i$. 
Finally, properties (iii--v) follow by \hyperlink{SC2}{\bf SC2-3} of $(\tilde{x}_{i},\tilde{x}_{j})$.
	
To prove claim \textit{\ref{claim_4}},
let us examine how the change of generators from $\mathcal{G}$ to $\mathcal{\tilde{G}}$ via $\mathcal{\overline{G}}$ modifies the total boundary matrix (see Paragraph \ref{point_change_bound_matrix}).	
We see that
\begin{align}
	\overline{d}_{j} &=(0,\dots,0,\underset{i\text{-th}}{1}, 0, \dots,0)^{T} \label{column_j}
	\\
 	\overline{d}^{k} &= {d}^{k} - d_{j}^{k}d^{i} \quad \qquad \text{ for } k\neq i \label{row_r} 
    \\
 	\overline{d}_{i} & = 0 \label{column_i_zero}
\end{align}
The equations (\ref{column_j}) and (\ref{row_r}) are induced by the change of generators (\ref{set_generators_bar}), and equation (\ref{column_i_zero}) follows from $\delta^{2}=0$.
Indeed, $\delta^{2}=0$ if and only if $\overline{d}^{2}=0$, which implies $(\overline{d}^{2})^{k}_{j}=0$ for all $k$.
On the other hand, by construction we have $\overline{d}_{j}^{i}=1$ and $\overline{d}_{j}^{k}=0$ for $k\neq i$.
So, $(\overline{d}^{2})^{k}_{j} = 
\sum_{h}\overline{d}_{h}^{k}\overline{d}_{j}^{h} =
\overline{d}_{i}^{k}$ by (\ref{column_j}).
Hence, $\overline{d}_{i}^{k}=0$ for all $k$, and thus the column $\overline{d}_{i}$ is trivial.
	
Next, let $\tilde{d}$ be the matrix associated with $\delta$ with respect to $\mathcal{\tilde{G}}$.
The change of the generators from $\mathcal{\overline{G}}$ to $\mathcal{\tilde{G}}$ affects only the $j$-th row, so that $\tilde{d}^{j}=\overline{d}^{j}+\overline{d}^{i}_{k}\overline{d}^{k}$, and the columns $\tilde{d}_{k}$ with $k\neq j$ and $\overline{d}_{k}^{i}\neq 0$.
In particular, equations (\ref{column_j})-(\ref{column_i_zero}) hold also for $\tilde{d}$, and in addition we have:
\begin{align}
	\tilde{d}^{i} & = (0,\dots, 0, \underset{j\text{-th}}{1}, 0, \dots,0) \label{row_i} 
	\\
	\tilde{d}^{j} & = 0 \label{row_j_zero}
\end{align}
Equality (\ref{row_i}) follows by the choice of the generators.
Indeed, we have $\tilde{d}_{k} = \overline{d}_{k}-\overline{d}_{h}^{i}\overline{d}_{j}$ for $k\neq j$, and
thus, by (\ref{column_j}), $\tilde{d}_{k}^{i} = \overline{d}_{k}^{i}-\overline{d}_{k}^{i}\overline{d}_{j}^{i} = \overline{d}_{k}^{i}-\overline{d}_{k}^{i}=0$  for all $k\neq j$.
Equality (\ref{row_j_zero}) follows by similar arguments as in the proof of (\ref{column_i_zero}).
Claim \textit{\ref{claim_4}} now follows by observing that the total boundary matrix with respect to $\mathcal{G'}=\tilde{\mathcal{G}}\setminus\{\tilde{x}_i,\tilde{x}_{j}\}$ corresponds to $\tilde{d}$ with the $i$-th row and column and the $j$-th row and column deleted.
\end{proof}

We stress that in the proof we relied on the assumption $\delta^{2}=0$.
Without this assumption, the $i$-th column and the $j$-th row need not be zero and thus cannot be directly deleted.
This is why we cannot generalise such an argument to the decomposition of two-parameter persistence modules.

Moreover, we remark that, in actual computations, it is sufficient to perform only the row reduction to determine the boundary matrix of the complex $X'$ that remains after the splitting. 
Indeed, the column reduction is necessary for the proof of the correctness of the procedure but does not affect the retrieving of the boundary of $X'$ (see \cref{S:example}).
\medskip

If a generator is a cycle and does not contribute to the differential, its corresponding interval sphere can be directly split.

\begin{proposition}\label{prop_split_trivial}
	Let $X$ be a filtered chain complex and ${\mathcal G}=\{x_1,x_2,\ldots, x_{m}\}$ an ordered quasi-minimal set of generators of $X$.
	Let $d$ be the boundary matrix associated with $\delta$ with respect to ${\mathcal G}$.
	For each  generator $x_{i}$ in $\mathcal{G}$ such that \text{(a)} $d^{i}=0$, and \text{(b)} $d_{i}=0$, it holds that:
	\begin{enumerate}
		\item $X$ is isomorphic to $\intS{s,\infty}\oplus X'$ for $n=\deg(x_{i})$, with $s=\ent(x_i)$, and $X'$ is the filtered chain complex whose set of generators is $\mathcal{G'}=\mathcal{G}\setminus\{x_{i}\}$, with same degrees and entrance times of $\mathcal{G}$, and differential induced by $\delta$;
		\item The $(m-1)\times(m-1)$ total boundary matrix $d'$ of $X'$ with respect to $\mathcal{G'}$ is obtained from the $m\times m$ total boundary matrix $d$ of $X$ with respect to $\mathcal{G}$ by removing its $i$-th row and column.
	\end{enumerate} 
\end{proposition}

\begin{proof}
	Consider the following exact sequence:
	\begin{equation}\label{exact_sequence}
	\begin{tikzcd}[column sep=0.5cm]
	0 \ar[r]
	& \intS{s,\infty}\ar[r, "\iota"]
	& X\ar[r]
	& X'\cong \faktor{X}{\Img(\iota)}\ar[r]
	& 0
	\end{tikzcd}
	\end{equation}
	where $\iota$ is such that $1\in (\intS{s,\infty})_s^n \mapsto x_{i}$, with $s=\ent(x_i)$.
	The hypotheses \textit{(a)} and \textit{(b)} are equivalent to fact that the differential of $x_{i}$ is trivial and $x_{i}\not\in \Img(\delta)$. 
	Thus, they guarantee the existence of a retract $\rho\colon X\to \intS{s,\infty}$.
	We can then apply the characterisation of split exact sequences \cite[Prop.\ 4.3]{maclane_homology} to (\ref{exact_sequence}) and conclude that $X\cong\intS{s,\infty}\oplus X'$, proving the first claim.
	The second claim follows immediately since $d^{i}=d_{i}=0$ by hypothesis.
\end{proof}

We are now ready to present the pseudo-code for the decomposition of filtered chain complexes. 
Algorithm~\ref{algorithm_decomposition_pickpair} takes as input an ordered quasi-minimal set of generators of $X$ and returns as output the list of interval sphere direct summands of $X$. 

\begin{center}
\begin{algorithm}[H]\label{algorithm_decomposition_pickpair}
	\caption{Interval sphere decomposition of filtered chain complexes} 
	\DontPrintSemicolon
	\KwIn{Totally ordered quasi-minimal set of generators $\mathcal{G}$ of $X$, labelled by degree and entrance time}
	\KwOut{List of interval spheres}
	List $= \emptyset$ \\
	$d = $ \texttt{BOUNDARY}$(\mathcal{G})$ \label{line_boundary}
	\\
	\While{$\exists \ i $ such that $d^{i}\neq 0$}{
		$i,j =$ \hyperlink{PAIR}{\texttt{PAIR}}$(d, i)$ \\
		Append $\intS[\dg(i)]{\ent(i),\ent(j)}$ to List \\
		\textnormal{\hyperlink{SPLIT}{\texttt{SPLIT}}$(i,j,d)$} 
	}
	\For{all indices $i$ of remaining rows in $d$}{
		Append $\intS[\dg(i)]{\ent(i),\infty}$ to List
	}
	Return List
\end{algorithm}
\end{center}

Algorithm~\ref{algorithm_decomposition_pickpair} hinges on three methods: \texttt{BOUNDARY}, \hyperlink{SPLIT}{\texttt{SPLIT}}, and \hyperlink{PAIR}{\texttt{PAIR}}. 
\texttt{BOUNDARY} generates the boundary matrix $d$ from a minimal set of generators and is not discussed here as it depends on the data structure at hand. 
The boundary matrix $d$ is then a global variable and will be updated at each call of \hyperlink{SPLIT}{\texttt{SPLIT}}. 
\hyperlink{SPLIT}{\texttt{SPLIT}} performs the splitting by reducing the boundary matrix, as stated in \cref{prop_split}(\textit{4}).
\hyperlink{PAIR}{\texttt{PAIR}} selects the pair of generators to be split following the constructive proof of \cref{prop_existence_xi_xj}, which guarantees that the while loop eventually terminates.
We underline that \hyperlink{PAIR}{\texttt{PAIR}} does not need columns and rows of $d$ to be sorted by an a-priori-chosen order, differently to what happens for usual persistence algorithms. 
For actual implementation, the most efficient approach is likely to order the generators as in the SPA (see \cref{sec_standard_alg}). 
Indeed, with that order, the method \hyperlink{PAIR}{\texttt{PAIR}} is much more efficient than the pseudo-code we present below. 
However, from a theoretical point of view, we show that we can reduce the boundary matrices and obtain the barcode decomposition without an a-priori-fixed order.
Finally, we note that \cref{prop_existence_xi_xj}, and consequently the \hyperlink{PAIR}{\texttt{PAIR}} method, would also work by exchanging the role of rows and columns, i.e.\ by first picking an arbitrary column $j$, and then looking for a row $i$ to pair it with. 

\begin{center}
	\begin{method}[H]\label{split_algorithm}
	\hypertarget{SPLIT}{}
		\caption{{\small \texttt{SPLIT}$(i,j, d)$}} 
		\DontPrintSemicolon
		\KwIn{Boundary matrix $d$, indices $i$, $j$}
		\KwOut{Updated boundary matrix $d$}
		\For{$k$ with $d_{j}^{k}\neq0$}{
			Add to the $k$-th row the $i$-th row multiplied by $-({d_{j}^{i})^{-1}d_{j}^{k}}$ \label{add_i_row_to_k_row_alg_0}
		}
		Delete row $j$ from $d$ \label{delete_row_j} \\
		Delete row $i$ and column $j$ from $d$ \label{delete_col_j_row_i} \\
		Delete column $i$ from $d$ \label{delete_col_i}		
	\end{method}
\end{center}

\begin{center}
	\begin{method}[H]\label{algorithm_pick&pair}
	\hypertarget{PAIR}{}
		\caption{\texttt{PAIR}$(d, i)$} 
		\DontPrintSemicolon 
		\KwIn{Boundary matrix $d$, index $i$ of a non-zero row of $d$}
		\KwOut{$i,j$ satisfying \hyperlink{SC1}{\bf SC1-2-3}}
		Pick $j\in \{ \argmin_{h \mid d_{h}^{i}\neq 0} \ent(h)\}$ \label{pair_initial_pick} \\
		\While{\textsc{True}\label{line_pair_until}}{\label{line_pair_1}
		    \If{$i\in \{ \argmax_{h \mid d_{j}^{h}\neq 0} \ent(h)\}$ \label{pair_row_condition}}{
		    Return $i,j$
		    }
		    \Else{
		    Pick $k\in \{ \argmax_{h \mid d_{j}^{h}\neq 0} \ent(h)\}$ \\
		    $i = k$
		    }
			\If{$j\in \{ \argmin_{h \mid d_{h}^{i}\neq 0} \ent(h)\}$}{
		    Return $i,j$
		    }
		    \Else{
		    Pick $k\in \{ \argmin_{h \mid d_{h}^{i}\neq 0} \ent(h)\}$ \\
		    $j = k$
		    }
		}
	\end{method}
\end{center}

In conclusion, \cref{prop_split} and \cref{prop_split_trivial} guarantee the overall correctness of Algorithm~\ref{algorithm_decomposition_pickpair} to obtain the decomposition of a filtered chain complex, and hence, as a by-product, the barcode decomposition of a one-parameter persistence module. 

\subsection{Examples}\label{S:example}

To illustrate the above methods, we provide two running examples built on the Vietoris--Rips complexes of two point clouds. 
Consider the $l_{\infty}$ metric on the following two spaces: $(a)$ four points on a line and $(b)$ four points in a grid. 
We call these points $x,y,z,w$ as illustrated below:

\begin{figure}[ht!]
    \centering
    \subfloat[Four points along a line]{
    \begin{tikzpicture}[scale=0.45,y=3cm, x=3cm,>=latex',font=\sffamily]
    %axis
    \draw[-,thick] (-0.1,0.75) -- coordinate (x axis mid) (2.6,0.75);
    \draw[-,white,thick] (-0.3,-0.2) -- coordinate (x axis mid) (2.6,-0.2);
    %points
    \foreach \coor in {0.5, 1, 1.5, 2}{
    \fill[Bnavy] (\coor,0.75) circle (6.5pt);}
    \node at (0.5,1) {x};
    \node at (1,1) {y};
    \node at (1.5,1) {z};
    \node at (2,1) {w};
    \end{tikzpicture}
    \label{F_line_points}} 
    $\quad$
    %--------------------------------------------------------------------------------
    \subfloat[Four points in a grid]{%
    \begin{tikzpicture}[scale=0.45,y=3cm, x=3cm,>=latex',font=\sffamily]
    %horizontal lines
    \draw[-,thick] (-0.1,0.5) -- coordinate (x axis mid) (1.6,0.5);
    \draw[-,thick] (-0.1,1) -- coordinate (x axis mid) (1.6,1);
    %vertical lines
    \draw[-,thick] (0.5,-0.1) -- coordinate (x axis mid) (0.5,1.6);
    \draw[-,thick] (1,-0.1) -- coordinate (x axis mid) (1,1.6);
    \draw[-,white,thick] (-1,-0.2) -- coordinate (x axis mid) (-1,1.6);
    %points
    \foreach \coordx in {0.5, 1}{
	\foreach \coordy in {0.5, 1}{
	\fill[Bnavy] (\coordx,\coordy) circle (6.5pt);}}
	\node at (0.3,0.3) {x};
    \node at (0.3,1.2) {y};
    \node at (1.2,1.2) {z};
    \node at (1.2,0.3) {w};
    \end{tikzpicture}
    \label{F_grid_points}}
\end{figure}

The corresponding distance matrices are as follows (hereafter, red colour is used for the case $(a)$ and blue colour for $(b)$):

\[
D^{\textcolor{red}{a}}= 
\begin{blockarray}{ccccc} 
& x & y & z & w \\
\begin{block}{c[cccc]}
x & 0 & 1 & 2 & 3 \\
y & 1 & 0 & 1 & 2 \\
z & 2 & 1 & 0 & 1 \\
w & 3 & 2 & 1 & 0 \\
\end{block}
\end{blockarray}
,
\quad 
D^{\textcolor{blue}{b}} =
\begin{blockarray}{ccccc} 
& x & y & z & w \\
\begin{block}{c[cccc]}
x & 0 & 1 & 1 & 1 \\
y & 1 & 0 & 1 & 1 \\
z & 1 & 1 & 0 & 1 \\
w & 1 & 1 & 1 & 0 \\
\end{block}
\end{blockarray}
\, .
\]

We illustrate the decomposition Algorithm \ref{algorithm_decomposition_pickpair} for the filtered chain complexes whose quasi-minimal sets of generators are given by the simplices of the Vietoris--Rips complexes on the two finite metric spaces given by these two matrices. 
First, we find a pair of generators to split using \hyperlink{PAIR}{\texttt{PAIR}}, and then we prune them off according to \hyperlink{SPLIT}{\texttt{SPLIT}}. 
As a field, we choose to work in  $\mathbb{Z}/2\mathbb{Z}$.
For simplicity, we will not display the total boundary matrices but only the block of degree 2. 
As a consequence, we will not perform Line \ref{delete_row_j} and Line \ref{delete_col_i} of Method \ref{split_algorithm} and we will shrink the matrix by only one column and one row.

As by Line \ref{line_boundary} of Algorithm \ref{algorithm_decomposition_pickpair}, we first build the boundary matrices $\partial_2^{\textcolor{red}{a}}$ and $\partial_2^{\textcolor{blue}{b}}$, whose columns correspond to the generators of degree $2$ and whose rows correspond to generators of degree $1$.
Observe that, since by \cref{prop_split} we can use any total order, we can pick the lexicographic order for the simplices so that the matrices have the same elements (i.e.\ $\partial_2^{\textcolor{red}{a}}=\partial_2^{\textcolor{blue}{b}}$ --- even if the labels of the rows and columns given by the entrance times will differ.
In the following matrix, the most-external labels below denote the simplices, and the innermost labels denote their entrance times.

\[
\partial_2^{\textcolor{red}{a}} \ =
\begin{blockarray}{ccccccc} 
&& xyz & xyw & xzw & yzw \\
&& \textcolor{blue}{1} & \textcolor{blue}{1} & \textcolor{blue}{1} &  \textcolor{blue}{1} & \text{\textcolor{blue}{Grid}} \\
\begin{block}{cc(cccc)c}
xy & \textcolor{red}{1} & 1 & 1 & 0 & 0 & \textcolor{blue}{1} \\
xz & \textcolor{red}{2} & 1 & 0 & 1 & 0 & \textcolor{blue}{1} \\
xw & \textcolor{red}{3} & 0 & 1 & 1 & 0 & \textcolor{blue}{1} \\
yz & \textcolor{red}{1} & 1 & 0 & 0 & 1 & \textcolor{blue}{1} \\
yw & \textcolor{red}{2} & 0 & 1 & 0 & 1 & \textcolor{blue}{1} \\
zw & \textcolor{red}{1} & 0 & 0 & 1 & 1 & \textcolor{blue}{1} \\
\end{block}
&\text{\textcolor{red}{Line}} & \textcolor{red}{2} & \textcolor{red}{3} & \textcolor{red}{3} & \textcolor{red}{2} &
\end{blockarray}
= \ \partial_2^{\textcolor{blue}{b}}
\]

Let us now use \hyperlink{PAIR}{\texttt{PAIR}} to find a row and a column to split. 
As required, we begin with a non-zero row: we choose the top one.
We now need to look at the entrance time labels of the non-zero elements in row $xy$ and choose one of the smallest ones (Line \ref{pair_initial_pick} of \hyperlink{PAIR}{\texttt{PAIR}}). 
We pick column $xyz$ for both matrices. 
Next, we need to check if row $xy$ has one of the greatest entrance times among the rows of the non-zero elements in column $xyz$. 
This is true for $(b)$ but not for $(a)$, as the one with the greatest entrance time is row $xz$. 
Thus, we exit the loop according to Line \ref{pair_row_condition} of \hyperlink{PAIR}{\texttt{PAIR}} with the generators $(xy,xyz)$ for $(b)$ and we repeat the procedure for $(a)$. 
Taking $xz$ in place of $xy$, we see that condition of Line \ref{pair_row_condition} of \hyperlink{PAIR}{\texttt{PAIR}} is met by the pair of simplices $(xz,xyz)$ for $(a)$.
Thus, the routine \hyperlink{PAIR}{\texttt{PAIR}} ends, and we have the generators of the interval sphere to be split also for $(a)$, namely $xz$ and $xyz$.

Next, we apply the \hyperlink{SPLIT}{\texttt{SPLIT}} method. 
This means that we first need to add row $xz$ to rows $xy$ and $yz$ in $\partial_2^a$, and row $xy$ to rows $xz$ and $yz$ in $\partial_2^b$ (Line \ref{add_i_row_to_k_row_alg_0} of Method \ref{split_algorithm}).
The matrices after these operations turns out to be:

\[
\partial_2^{\textcolor{red}{a}} \ =
\begin{blockarray}{cccccc} 
&& xyz & xyw & xzw & yzw \\
\begin{block}{cc(cccc)}
xy & \textcolor{red}{1} & 0 & 1 & 1 & 0 \\
xz & \textcolor{red}{2} & 1 & 0 & 1 & 0 \\
xw & \textcolor{red}{3} & 0 & 1 & 1 & 0 \\
yz & \textcolor{red}{1} & 0 & 0 & 1 & 1 \\
yw & \textcolor{red}{2} & 0 & 1 & 0 & 1 \\
zw & \textcolor{red}{1} & 0 & 0 & 1 & 1 \\
\end{block}
&\text{\textcolor{red}{Line}} & \textcolor{red}{2} & \textcolor{red}{3} & \textcolor{red}{3} & \textcolor{red}{2} 
\end{blockarray}
\quad
\partial_2^{\textcolor{blue}{b}} \ =
\begin{blockarray}{cccccc} 
& xyz & xyw & xzw & yzw \\
& \textcolor{blue}{1} & \textcolor{blue}{1} & \textcolor{blue}{1} &  \textcolor{blue}{1} & \text{\textcolor{blue}{Grid}} \\
\begin{block}{c(cccc)c}
xy & 1 & 1 & 0 & 0 & \textcolor{blue}{1} \\
xz & 0 & 1 & 1 & 0 & \textcolor{blue}{1} \\
xw & 0 & 1 & 1 & 0 & \textcolor{blue}{1} \\
yz & 0 & 1 & 0 & 1 & \textcolor{blue}{1} \\
yw & 0 & 1 & 0 & 1 & \textcolor{blue}{1} \\
zw & 0 & 0 & 1 & 1 & \textcolor{blue}{1} \\
\end{block}
\end{blockarray}
\]

At this point, we are in the position to remove row $xz$, resp. $xy$, and column $xyz$ from $\partial_2^a$, resp. $\partial_2^b$, thus shrinking the boundary matrices (Line \ref{delete_col_j_row_i} of \hyperlink{SPLIT}{\texttt{SPLIT}}). 
Finally, we obtain the interval sphere $\intS[1]{1,1}$ for the grid point-cloud and $\intS[1]{2,2}$ for the line one. 
Indeed, these are the entrance time labels for row $xz$, resp. $xy$, and column $xyz$ in the grid case (resp. the line case). 
Note that both these interval spheres have zero homology: their difference goes undetected at the homological level. 
The degree $2$ boundary matrices now are:

\[
\partial_2^{\textcolor{red}{a}}\  =
\begin{blockarray}{ccccc} 
&&  xyw & xzw & yzw \\
\begin{block}{cc(ccc)}
xy & \textcolor{red}{1} & 1 & 1 & 0 \\
xw & \textcolor{red}{3} & 1 & 1 & 0 \\
yz & \textcolor{red}{1} & 0 & 1 & 1 \\
yw & \textcolor{red}{2} & 1 & 0 & 1 \\
zw & \textcolor{red}{1} & 0 & 1 & 1 \\
\end{block}
&\text{\textcolor{red}{Line}} & \textcolor{red}{3} & \textcolor{red}{3} & \textcolor{red}{2}
\end{blockarray}
\qquad
\partial_2^{\textcolor{blue}{b}} \  =
\begin{blockarray}{ccccc} 
&  xyw & xzw & yzw \\
& \textcolor{blue}{1} & \textcolor{blue}{1} &  \textcolor{blue}{1} & \text{\textcolor{blue}{Grid}} \\
\begin{block}{c(ccc)c}
xy & 1 & 1 & 0 & \textcolor{blue}{1} \\
xw & 1 & 1 & 0 & \textcolor{blue}{1} \\
yz & 1 & 0 & 1 & \textcolor{blue}{1} \\
yw & 1 & 0 & 1 & \textcolor{blue}{1} \\
zw & 0 & 1 & 1 & \textcolor{blue}{1} \\
\end{block}
\end{blockarray}
\]
and we can look for the next pair to split. 
After completely running Algorithm \ref{algorithm_decomposition_pickpair}, also in the other degrees, the decomposition of the filtered chain complexes turns out to be 
\[
\text{(a)} \ 
\bigoplus_{i=1}^{3}\intS[0]{0,1} \oplus \bigoplus_{i=1}^{2} \intS[1]{2,2} \oplus \intS[1]{3,3} \oplus \intS[2]{3,3} \, , \quad
\text{(b)} \ 
\bigoplus_{i=1}^{3}\intS[0]{0,1} \oplus \bigoplus_{i=1}^{3}\intS[1]{1,1} \oplus \intS[2]{1,1} \, .
\]

\subsection{Filtered kernels}

In the previous section, we applied our algorithm on filtered Vietoris--Rips complexes. 
However Algorithm~\ref{algorithm_decomposition_pickpair} is more general and can be applied to arbitrary filtered chain complexes. 
To convey the idea, in this section, we illustrate one such application. 
We consider the kernels of morphisms between filtered chain complexes, which we call \define{filtered kernels}.
Note that the homology of filtered kernels is generally different from kernel persistence \cite{ker-img-coker,morozov-thesis, casas2020distributing} as $\ker H f \not\cong H \ker f$ in general.
 
\begin{point}\label{P:simplicial_map}{\em Simplicial maps and induced morphisms of filtered chain complexes.}
Let $f\colon K \to L$ be a simplicial map, and assume filtrations on $K$ and $L$, respectively, are given. 
To obtain quasi-minimal set of generators, as in Paragraph \ref{point_generators}, we label each simplex $\sigma$ with a degree $\deg(\sigma)$ (given by its simplicial dimension) and an entrance time $\ent(\sigma)$ (given by the filtration). 
Since $f$ is simplicial, for every simplex $\sigma \in K$, we have $\ent(\sigma) \geq \ent(f(\sigma))$ and $\deg(\sigma) \geq \deg(f(\sigma))$. 
We can now take $X$ and $Y$ to be the two filtered chain complexes induced by the given filtrations on $K$ and $L$, respectively. 
Then the map $f$ induces a morphism of filtered chain complexes, called again $f$, defined by 
\begin{equation*}
    f_{n}^{t} \coloneqq
	\begin{cases*}
	X^{\ent(f(\sigma))\leq \ent(\sigma)}(f(\sigma))  & if  $\dim f(\sigma) = \dim (\sigma) = n$ \\
	0           & otherwise
	\end{cases*} \, .
\end{equation*}
The filtered chain complex given by the kernel of $f$, hence called filtered kernel and denoted by $\ker f$, is defined by $(\ker f)_n^t \coloneqq \ker f_n^t$ for all $n\in\mathbb{N}$ and $t\in[0,\infty)$, with the transition morphisms and boundary maps being those induced by $X$. 
Note that we can analogously define the filtered image $\Img f$ as $(\Img f)_n^t \coloneqq \Img f_n^t$ and the transition morphisms and boundary maps induced by $Y$.
As for filtered kernels, this is in general different from image persistence \cite{ker-img-coker, morozov-thesis, casas2020distributing}. 
\end{point}

Our goal is to use  quasi-minimal sets of generators of $K$ and $L$, precisely their simplices, to describe a quasi-minimal set of generators of the filtered kernel of $f$, in the case when $f$ admits a section. 
This way, Algorithm~\ref{algorithm_decomposition_pickpair} can be applied to compute a decomposition of $\ker f$.

\begin{point}{\em Quasi-minimal set of generators for $\ker f$.}
We now fix a simplicial map $f\colon K\to L$ between two filtered simplicial complexes admitting a section $s \colon L \to K$ (see \cite{section} for conditions that guarantee this requirement). 
Note that, for such $s$, for every $\tau \in L$, $\ent(\tau)=\ent(s(\tau))$ and $\deg(\tau)=\deg(s(\tau))$.
The section $s$ induces a morphism of filtered chain complexes denoted again by $s$ (Paragraph \ref{P:simplicial_map}). 
We now define a quasi-minimal set of generators for $\ker f$ (see Paragraph \ref{point_generators}): for every $n\in\mathbb{N}$ and $t \in [0,\infty)$, a base of $\ker f_n^t$ is given by
\begin{equation*}
    V = \{ \sigma - s f(\sigma) \ \vert \ \deg(\sigma) = n, \ \ent(\sigma)=t \text{ and }  \sigma \not\in\Img(s) \} \, .
\end{equation*}
Indeed, since $f$ is an epimorphism and $s$ its section, we have that the dimension of the space generated by $V$ is equal to
\[ 
\dim \text{span}(V) = \dim X_n^t - \dim(\Img s_n^t) = \dim X_n^t - \dim(\Img f_n^t) = \dim(\ker f_n^t) \, ,
\]
where the first equality holds because $\sigma \not\in \Img(s)$ implies $\sigma\neq sf(\sigma)$.
Thus, it is left to prove that the elements in $V$ are linearly independent as vectors. 
But this follows directly from the fact that each element is given by the difference between $\sigma \not\in \Img(s)$ and something which is in the image of $s$. 
\end{point}

We can thus easily compute the quasi-minimal set of generators of $\ker f$ from the quasi-minimal sets of generators of $K$ and $L$, once a section $s$ is known. 
Next, we show that retrieving the boundary matrix of $\ker f$ requires almost no computation.

\begin{point}{\em Boundary matrix of $\ker f$.} 
To retrieve the boundary matrix of $\ker f$, we denote by $\ker f_n$ the parametrised vector space generated by all the elements in $\ker f$ with degree $n$. 
Consider a quasi-minimal generator $\sigma-sf(\sigma)$, with $\sigma \in K$. 
Assume the boundary of $\sigma$ in $X$ is $\sum_{i=0}^{n}(-1)^i \tau_i$. 
Then the boundary of $\sigma - sf(\sigma)$ is given by:
\begin{align*}
    \delta(\sigma - sf(\sigma)) & = \delta\sigma - sf(\delta\sigma) =
    \sum_{i=0}^{n} (-1)^i\tau_i - sf(\sum_{i=0}^{n}(-1)^i\tau_i) = \sum_{i=0}^{n}(-1)^i\left(\tau_i -sf(\tau_i)\right) \, ,
\end{align*}
where $s$ and $f$ are considered as morphisms of filtered chain complexes (Paragraph \ref{P:simplicial_map}).
Since the elements of the form $\left(\tau_i -sf(\tau_i)\right)$ either belong to a basis of $\ker f_{n-1}$ (if $\tau_i \neq sf(\tau_i)$) or are equal to zero, the boundary matrix $\partial$ of $\ker f$ is obtained from that of $X$ removing the columns and rows corresponding to simplices $\sigma \in K$ for which $\sigma = sf(\sigma)$. 
\end{point}

In conclusion, with our decomposition algorithm, we can efficiently decompose $\ker f$, which is a filtered chain complex that does not come from a simplicial complex, showing that the algorithm can be applied in more general scenarios.

\section{The structural geometry behind the decomposition}\label{sec_alg_analysis}

One of the contributions of our algorithm is to highlight the geometry behind the SPA (see \cref{sec_standard_alg}). 
Indeed, this algorithm, together with its variants such as
\cite{ripser_paper,clearcompress,phat_paper, dualities, Edel-Olsb}, works at the chain complex level, but it is usually described at the persistence module level.

\begin{point}\label{point_standard}
{\em Relation with persistence algorithms.} 
Each interval sphere $\intS{s,e}$ has two generators, namely the one of the sphere $\Sp[n]$ at step $s$ and one of the disk $\D[n+1]$ capping such sphere at step $e$. 
These generators form the pair that in the SPA is given by the pivot pairing.
As explained in Paragraph \ref{relation_int_sphere_int_modules}, by forgetting (resp. retaining) the interval spheres whose generators have the same entrance time, i.e. the contractible parts, we can retrieve the interval modules (resp. spheres) decomposition from the interval spheres (resp. modules) decomposition. 
Thus, our algorithm explains the geometry behind the SPA, as we explicitly split the corresponding spheres and disks.

Moreover, we can now see that the clear and compress optimisations are implicit methods to quotient out both the generators: in the former the identification of the generator of $\D[n+1]$ allows us to skip the reduction of the one of $\Sp[n]$, while in the latter the opposite happens. 
In our algorithm, this is explicit: since the two generators generate the same indecomposable, we have to remove both in the decomposition.
We can read the compress and clear strategy in our pseudo-code at Line \ref{delete_row_j} and at Line \ref{delete_col_i} of Method~\ref{split_algorithm}, respectively.
Thus, our method provides a geometric explanation of the optimisations at the chain complex level: rather than mere strategic optimisations, they are structural.
\end{point}

\begin{point}\label{point_rowvscolumn}
	{\em Row versus column reductions.}
	The SPA computes the barcode by reducing the boundary matrix by column operations, thus performing operations on the chains. 
	On the other hand, the authors of \cite{dualities} obtain the barcode by reducing (still by column operations) the coboundary matrix, thus performing operations on the cochains. 
	According to \cite{phat_paper}, the use of cohomology with respect to homology turns out to be particularly convenient, when combined with the clear strategy, in the computation of the barcode of Vietoris--Rips complexes. 

    To properly dualise the optimisation obtained in cohomology by the clear strategy, the boundary matrix has to be reduced via row rather than column operations \cite{notes_pivot}. 
	This is precisely what Algorithm~\ref{algorithm_decomposition_pickpair} does.
	In doing so, when generators are primarily sorted by degree, the combination of the row reductions with the compress strategy achieves in homology the same advantages as in cohomology with the clear, for Vietoris--Rips complexes. 
     
	In \cite{ripser_paper}, an explicit calculation quantifies the computational speed-up that the column reduction with clear on the coboundary matrix provides in the barcode computation of Vietoris--Rips complexes. 
	We now display the analogous computation for row operations in homology.
	We underline that what is computed is the number of iterations. 
	Each iteration has to be multiplied by the reduction cost, which depends on the algorithm used. 
 
	Let $X$ be a filtered Vietoris--Rips complex, whose quasi-minimal set of generators is ordered first by degree and then by a total order of choice.
	Let $v$ be the number of generators in degree $0$, i.e. the number of vertices,  $N$ the maximal non-trivial degree of $X$, and $d_{n}$, for $1\leq n\leq N+1$, the blocks of the boundary matrix $d$ of $X$.
	To reduce $d$ without compress, we need to process one row per generator from degree $0$ to $N$.
	In total, this sums up to
	\begin{equation}\label{eq_number_iteration_tot}
	\displaystyle\sum_{n=0}^{N} \underbrace{\binom{v}{n+1}}_{\dim(C_{n}X)} =
	\displaystyle\sum_{n=0}^{N} \underbrace{\binom{v-1}{n}}_{\dim(B_{n-1}X)}+
	\displaystyle\sum_{n=0}^{N} \underbrace{\binom{v-1}{n+1}}_{\dim(Z_{n}X)}
	\, ,
	\end{equation}
	where the equation follows by the fact that the final step of a Vietoris--Rips filtration is a full-complex, the definition of homology, and the rank-nullity theorem, see also \cite[Sec. 3.3]{ripser_paper}.
	Note that $\dim(B_{n-1}X)$ is the number of negative generators of degree $n$.
	Applying the compress strategy, since the negative rows of $d_n$, for $n>0$, are removed when processing $d_{n-1}$, (\ref{eq_number_iteration_tot}) decreases to 
	\begin{equation*}\label{eq_number_row_red_clear}
	\underbrace{\binom{v-1}{0}}_{\dim(B_{-1}X)} +
	\displaystyle\sum_{n=0}^{N} \underbrace{\binom{v-1}{n+1}}_{\dim(Z_{n}X)}
	= 
	1 + \displaystyle\sum_{n=0}^{N}\binom{v-1}{n+1}
	= 
	1 + \displaystyle\sum_{n=1}^{N+1}\binom{v-1}{n}  
	\end{equation*}
	that coincides with the number of generators processed in \cite{ripser_paper}.
\end{point}

\begin{point}\label{remark_apparent_pairs}
	{\em Apparent and emergent pairs.} 
	According to \cite{ripser_paper}, in the SPA where rows and columns of the boundary matrix are ordered first by entrance time and then by degree, the pairs of indices $(i,j)$ such that the lowest non-zero element of the $j$-th column is in the $i$-th row, and the leftmost non-zero element of the $i$-th row is in the $j$-th column, are called \define{apparent pairs} (see \cite[Sec.\ 3.6]{ripser_paper} for an overview of the use of apparent pairs in the computation of persistent homology). 
	Thus, the apparent pairs are precisely given by pairs of generators $(x_{i},x_{j})$ that satisfy the split conditions \hyperlink{SC1}{\bf SC1-2-3} before the beginning of the reduction. 
	In other words, they are all and only the pairs that always exit the loop in \hyperlink{PAIR}{\texttt{PAIR}} at the first iteration.
    Note that the generators in an apparent pair depend on the chosen total order: there can be more than one pair of generators with the same degree and entrance time, and the order between them decides which one belongs to the apparent pair. 
	The interval decomposition is, clearly, independent from this choice.
	
	More generally, an \define{emergent (co)facet pair} is a pair of generators $(x_i,x_j)$ that get split and such that $x_i$ is the youngest facet of $x_j$ (resp. $x_j$ is the oldest cofacet of $x_i$). 
	Thus, emergent pairs are pairs of generators that immediately exit the loop in \hyperlink{PAIR}{\texttt{PAIR}}, when \hyperlink{PAIR}{\texttt{PAIR}} is called after at least one iteration.
\end{point}

\section{Discussion and conclusions}\label{sec_conclusions}

In this work, we analysed persistence algorithms based on matrix reductions from the standpoint of filtered chain complexes. 
Using this approach, we have been able to geometrically explain the decomposition as well as the clear and compress optimisations.
Nevertheless, several questions remain open.
For example, a more efficient implementation of the \hyperlink{PAIR}{\texttt{PAIR}} method, ideally tailored for different chain complexes, needs to be studied to see if it can provide some speed up in computation. 
One could also ask if, due to the independence of the order, there is a more efficient way to sort the generators to allow quick access in the memory or a faster way to calculate the (co)boundary. 

We also introduced filtered kernels and provided a setting in which they can be decomposed with our algorithm. 
This setting can be realised in different situations. 
For example, suppose we have a distance space $M$ that we want to enlarge by glueing non-empty fibers to its points. 
The resulting distance space $M'$ is called an extension of $M$
if the projection $f\colon M' \to M$ is $1$-Lipschitz. Any such extension leads to a split short exact sequence $0\to \ker\to X \to Y \to 0$, where $X$ and $Y$ are the filtered chain complexes built from filtrations on $M'$ and $M$, respectively. Thus,  in such a setting, the decomposition of $X$ can be obtained from the decompositions of $Y$ and $\ker $, where the latter may not come, in any natural way, from a filtered simplicial complex. 
Decompositions of filtered chain complexes of  all extensions of 
$M$ can be analysed in this way, without the need to recomputing the decomposition of $Y$, possibly leading to more effective calculations.

Yet another aspect to be investigated is the information carried by the points on the diagonal. 
There is some empirical evidence that these points are not always noise. 
In the example in \cref{S:example}, we see that interval modules do not distinguish between the two metric spaces, whereas interval spheres do. 
This indicates that filtered chain complexes retain more geometric information that would be lost applying homology. 
We also have that, in the case of a Vietoris--Rips complex, the latest-appearing interval sphere in degree $1$ has an interval of length zero, and the value of the starting point of the interval is the diameter of the point cloud.
For another example, the authors of \cite{diffraction} noted that the presence and multiplicity of points on and very near the diagonal of persistence diagrams obtained from studying the diffraction of different materials are related to the density of the materials (see Fig.\ 11 (A) and (B) in \cite{diffraction}). 
Moreover, in \cite{fasy2022faithful}, the zero-length intervals are necessary to ensure the proper reconstruction of the simplicial complex.
These different examples suggest that the zero-length intervals (which are the points on the diagonal of a persistence diagram) carry some meaningful information. 
We believe that they should be investigated in a more structured study, and we hope that our algorithm will be helpful to this end.

\paragraph{Acknowledgment.}
W.\ Chach\'olski was partially supported by VR, the Wallenberg AI, Autonomous System and Software Program (WASP) funded by Knut and Alice Wallenberg Foundation, and MultipleMS funded by the European Union under the Horizon 2020 program, grant agreement 733161, and dBRAIN collaborative project at digital futures at KTH.
This work was partially carried out by B. Giunti and C. Landi within the activities of ARCES (University of Bologna) and under the auspices of INdAM. 
B. Giunti benefited from the hospitality by the mathematics department of NTNU Trondheim and was partially supported by the Austrian Science Fund (FWF) grants number P29984-N35 and P33765-N. 
We thank Guillaume Houry, Michael Kerber, and Sara Scaramuccia for valuable discussions.
We also thank the referees for their feedback.

\end{document}